\documentclass[12pt,a4paper]{article}
\usepackage[T1]{fontenc}
\pdfoutput=1

\usepackage[utf8]{inputenc}
\usepackage[english]{babel}
\usepackage{amsmath}
\usepackage{cite}
\usepackage{amsfonts}
\usepackage{amssymb}
\usepackage{amsthm}
\usepackage{mathtools}
\usepackage[colorlinks=true,allcolors=blue]{hyperref}%
\usepackage[left=2.5cm,right=2.5cm,top=2.5cm,bottom=2.5cm]{geometry}
\usepackage{algorithm}
\usepackage{algpseudocode}
\usepackage{tabularx}
\usepackage{subfigure}
\usepackage{graphicx}
\usepackage{float}
\usepackage{color}
\usepackage[margin=1cm,font=small,format=hang,labelsep=period,labelfont=bf]{caption}
\usepackage{titling}
\usepackage{verbatim}
\usepackage{physics}
\usepackage{setspace}
\usepackage{authblk}
\usepackage{chemarrow}

\newtheorem{thm}{Theorem}[section]
\newtheorem{definition}[thm]{Definition}
\newtheorem{lemma}[thm]{Lemma}
\newtheorem{prop}[thm]{Proposition}
\newtheorem{cor}[thm]{Corollary}
\newtheorem{rem}[thm]{Remark}

\newcommand{\RR}{\mathcal{K}}

\usepackage{tikz}
\usetikzlibrary{graphs}
\usetikzlibrary{shadows}
\definecolor{myblue}{RGB}{240,240,255}
\definecolor{myshadow}{RGB}{200,200,200}
\tikzstyle{complex}=[shape=rectangle, draw=black, fill=myblue, drop shadow=myshadow]
\tikzset{every loop/.style={min distance=10mm,in=323,out=217,looseness=6}}

\makeatletter
\def\BState{\State\hskip-\ALG@thistlm}
\newcommand{\multiline}[1]{%
	\begin{tabularx}{\dimexpr\linewidth-\ALG@thistlm}[t]{@{}X@{}}
		#1
	\end{tabularx}
}
\makeatother
\algdef{SE}[SUBALG]{Indent}{EndIndent}{}{\algorithmicend\ }%
\algtext*{Indent}
\algtext*{EndIndent}

\author[1,2]{G\'abor Szederk\'enyi} 
\author[1]{Bernadett Ács}
\author[1]{György Lipták}
\author[1]{Mihály A. Vághy}
\affil[1]{\small P\'azm\'any P\'eter Catholic University, Faculty of Information Technology and Bionics, Práter u. 50/a,  H-1083 Budapest, Hungary} 
\affil[2]{\small Systems and Control Laboratory, ELKH Institute for Computer Science and Control (SZTAKI), Kende u. 13-17, H-1111 Budapest, Hungary}
\affil[ ]{\small e-mail: szederkenyi@itk.ppke.hu}
\title{\vspace*{2.5cm} Persistence and stability of a class of kinetic compartmental models}
\date{}

\begin{document}

\maketitle

%
%
%


\begin{abstract}
In this paper we show that the dynamics of a class of kinetic compartmental models with bounded capacities, monotone reaction rates and a strongly connected interconnection structure is persistent. The result is based on the chemical reaction network (CRN) and the corresponding Petri net representation of the system. For the persistence analysis, it is shown that all siphons in the Petri net of the studied model class can be characterized efficiently. Additionally, the existence and stability of equilibria are also analyzed building on the persistence and the theory of general compartmental systems. The obtained results can be applied in the analysis of general kinetic models based on the simple exclusion principle.
\end{abstract}


%

%
%
%
%
\section{Introduction}
Nonnegative systems form an important subclass within dynamical systems having the property that the nonnegative orthant is invariant with respect to the dynamics. The practical motivation for developing the theory of nonnegative systems is the fact that there are several application fields such as chemistry, biology, population and disease dynamics, where in many cases the state variables of the models in the original physical coordinates are nonnegative \cite{Haddad2010}. 
Compartmental models are used to describe the change of distribution of objects (e.g., molecules or particles) 
among different storage compartments in time \cite{anderson2013compartmental}. Compartments can be physically distinct subsystems such as interconnected containers,
but they can also represent disjoint states like different stages of diseases in the case of epidemic models \cite{brauer2008compartmental}. Since the natural state variables in compartmental systems correspond to amounts of materials, numbers of molecules (or to their ratios, concentrations), these models belong to the nonnegative system class. The fundamental properties of compartmental models have been intensively studied in the literature. The observability, controllability, realizability and identifiability of compartmental systems are summarized in \cite{brown1980compartmental} focusing mainly on linear models. The analytic solution of linear compartmental ODEs is studied in \cite{garcia2012linear} in a kinetic context. A fundamental reference on the qualitative analysis of a wide class of general nonlinear compartmental models is \cite{jacquez1993qualitative}, where important results can be found on the structure of equilibria and stability.

It is known that most compartmental models can be represented in the form of kinetic systems also called chemical reaction networks (CRNs), where the dynamics can be formally realized by a set of reactions with appropriate complexes and reaction rates \cite{Erdi1989,hofmeyr2020kinetic}. Although kinetic models are originated from physical chemistry, they have been highly generalized in a mathematical sense (see, e.g. \cite{muller2012generalized,anderson2020classes,hernandez2021positive}), widening their application possibilities even to non-chemical processes as general descriptors of nonlinear dynamics. Chemical reaction network theory (CRNT) is a dynamically improving research field with strong results on the relations between the reaction graph structure and the qualitative properties of the kinetic dynamics \cite{Feinberg2019}. Persistence analysis is a problem of central importance in CRNT, for instance, it is a key property for proving global asymptotic stability of complex balanced networks \cite{Anderson2011,craciun2013persistence,Craciun2015}. 

The application of discrete structures and graph theory is an essential tool in the modeling and analysis of chemical phenomena \cite{balaban1985applications,Burch2019,anderson2021deficiency}. If the molecules and reactions in a kinetic system are tracked individually, the CRN can be considered as a discrete event system which is equivalent to a Petri net \cite{Angeli2009,pauleve2014dynamical}. A fundamental and theoretically deep result of CRNT is published in \cite{angeli2007petri}, where conditions are given for the persistence of continuous time CRNs using the graph structure of their Petri net representation and conserved quantities in the dynamics. These results were further generalized in \cite{angeli2011persistence} to time-dependent open systems, where reaction rates can be time-dependent and there are in and outflows. 

The model class we study is related to the so-called simple exclusion principle known from the theory of Markov Processes \cite{komorowski2012simple} in the sense that particles can move along a directed graph (called the compartmental graph). It is also assumed that the capacity of the compartments is bounded. Therefore, transition is only possible if there are available particles in the donor compartment, and also free space in the recipient compartment. A well-known application of the simple exclusion principle is the class of ribosome flow models (RFMs) \cite{reuveni2011genome} capturing key features of the translation process. Numerous valuable analysis results have been developed for RFMs, we can only mention a few. In \cite{margaliot2012stability} it is shown that RFMs with a tube-like structure have a unique asymptotically stable equilibrium point within the invariant domain of their dynamics. RFMs with a ring topology are studied in \cite{raveh2015ribosome} where it is shown that trajectories converge to equilibria within the compatibility classes of the state space defined by the initial conditions. The dynamics and stability of RFMs under periodic excitation is analyzed in \cite{margaliot2014entrainment}. We also mention that ODE models with essentially the same structure can be obtained by an appropriate finite volume discretization of hyperbolic partial differential equations describing the flow of material or vehicles \cite{liptak2021traffic}.

The structure of the paper is the following. In Section 2, we introduce the basic notions and known results for the ODE and Petri net representation of kinetic models. Section 3 describes the studied kinetic compartmental model class. Section 4 contains the persistence analysis results through the characterization of siphons, while the stability results are summarized in Section 5. Finally, the brief summary of the results is given in Section 6.

%
%
\section{Background and notations}
In this section, we introduce the class of kinetic systems and their representation in the form of Petri nets. Throughout the paper, we will use the following notations.

\subsection{Kinetic systems}\label{subsec:kinetic}
For the characterization of kinetic systems (also called chemical reaction networks or briefly, CRNs), we will use the notations used in \cite{Feinberg2019}, where more details can be found. A kinetic model contains $M$ species denoted by $\mathcal{X}=\{{X}_1,\dots,{X}_M\}$, and the corresponding species vector is given as $X=[{X}_1~\dots~{X}_M]^T$. Species are transformed into each other through \textit{elementary reaction steps} of the form
 \begin{align}
 C_j \rightarrow C'_j, \quad j=1,\dots,R \label{eq:Reactions}
 \end{align}
where $C_j=y_j^T {X}$ and $C'_j={y'_j}^T {X}$ are the \textit{complexes} with the \textit{stoichiometric coefficient vectors} $y_j,y'_j\in\overline{\mathbb{Z}}_{+}^M$ for $j=1,\dots,R$. The transformation shown in Eq. \eqref{eq:Reactions} means that during an elementary reaction step between the \textit{reactant complex} $C_j$ and \textit{product complex} $C_j'$, $[y_j]_i$ molecules of species ${X}_i$ are consumed, and $[y'_j]_i$ molecules of ${X}_i$ are produced for $i=1,\dots,M$. The reaction \eqref{eq:Reactions} is called an \textit{input (output) reaction of species $X_i$} if $[y_j']_i>0$ ($[y_j]_i>0$). 

The directed graph containing the complexes as vertices and reactions as directed edges is called the \textit{reaction graph} of a CRN. A directed graph is \textit{strongly connected} if there exists a directed path between any pair of its vertices in both directions. A \textit{strong component} of a directed graph is a maximal strongly connected subgraph. A \textit{weakly connected component} of a directed graph is a subgraph where all vertices are connected to each other by some (not necessarily directed) path. A reaction graph is called \textit{weakly reversible} if each weakly connected component of it is a strong component. Weak reversibility is equivalent to the property that each directed edge (reaction) is a part of a directed cycle in the reaction graph.

Let $x(t)\in\overline{\mathbb{R}}^M_{+}$ denote the state vector corresponding to ${X}$ for any $t\ge 0$ (in a chemical context, $x$ is the vector of concentrations of the species in ${X}$). Then the ODEs describing the evolution of $x$ in the kinetic system containing the reactions \eqref{eq:Reactions} are given by
\begin{align}\label{eq:general_kinetic}
\dot{x}=\sum_{i=1}^R \RR_i(x)[y_i' - y_i], \quad x(0)\in\overline{\mathbb{R}}^M_+    
\end{align}
where $\RR_i:\overline{\mathbb{R}}_+^M   \longrightarrow \overline{\mathbb{R}}_{+}$ is the \textit{rate function} corresponding to reaction step $i$, determining the velocity of the transformation \cite{Feinberg2019}. For the rate functions, we assume the following for $i=1,\dots,R$:
\begin{itemize}
\item[(A1)] $\RR_i$ is differentiable,
\item[(A2)] $\dfrac{\partial \RR_i (x)}{\partial x_j}\ge 0$ if  $[y_i]_j>0$, and $\dfrac{\partial \RR_i (x)}{\partial x_j} = 0$ if $[y_i]_j=0$,
\item[(A3)] $\RR_i(x)=0$ whenever $x_j=0$ such that $j\in\text{supp}(y_i)$.
\end{itemize}
The above properties guarantee the local existence and uniqueness of the solutions as well as the invariance of the nonnegative orthant for the dynamics in Eq. \eqref{eq:general_kinetic}. 
From now on, a reaction from complex $C_i$ to complex $C_i'$ with rate function $\RR_i$ will be denoted as 
\begin{align}
C_i \autorightarrow{$\RR_i$}{} C_i' 
\end{align}
The dynamics of a kinetic system \eqref{eq:general_kinetic} is called \textit{persistent} if no trajectory that starts in the positive orthant has an omega-limit point on the boundary of $\mathbb{R}^M_{+}$.

A set of nonlinear ODEs given as $\dot{x}=f(x)$
is called \textit{kinetic} if it can be written in the form \eqref{eq:general_kinetic} with appropriate rate functions $\mathcal{K}_i$. We remark that the representation $\eqref{eq:general_kinetic}$ of a kinetic ODE is generally non-unique even if the rate functions are polynomial, and assumed to be fixed \cite{Acs2016}. 

An important special case in the theory of CRNs is \textit{mass action kinetics} when the rate function is given in the following monomial form
\begin{equation}\label{eq:MassAction}
\RR_i(x)=k_i \prod_{j=1}^M x_j^{[y_i]_j},~~i=1,\dots,R
\end{equation}
where $k_i>0$ for $i=1,\dots,R$ are the \textit{reaction rate coefficients}.

A \textit{positive linear conserved quantity} (or positive linear first integral) for a CRN is defined as $c^T x$ for which $c^T \dot{x}(t)=0$ for $t\ge 0$, where $c\in\overline{\mathbb{R}}^M_{+}$ and $c\ne 0$. We say that a set of species $\{X_{i_1},\dots,X_{i_k} \}\subseteq\mathcal{X}$ \textit{defines a positive linear conserved quantity} if there exists $c\in\overline{\mathbb{R}}^k_{+}$ for which $\sum_{j=1}^k c_j \dot{x}_{i_j}(t)=0$ for $t\ge 0$.

\subsubsection*{Example 1}
Consider the following CRN given by $\mathcal{X}=\{X_1,\dots,X_6 \}$ and the reactions
\begin{align}
\begin{split}\label{eq:Ex1_react}
R_1: & ~~X_1 + X_5 \autorightarrow{$\RR_1$}{} X_2 + X_4  \\
R_2: & ~~X_2 + X_6  \autorightarrow{$\RR_2$}{} X_3 + X_5 \\
R_3: & ~~X_3 + X_4 \autorightarrow{$\RR_3$}{} X_1 + X_6
\end{split}
\end{align}
Furthermore, assume that the reaction rates obey the mass action kinetics described in Eq. \eqref{eq:MassAction}, i.e.
\begin{align}
\RR_1(x)=k_1 x_1 x_2,~\RR_2(x)=k_2 x_2 x_6,~\RR_3(x)=k_3 x_3 x_4
\end{align}
Then according to Eq. \eqref{eq:general_kinetic}, the ODEs of the system can be written as
\begin{align}
\begin{split}\label{eq:Ex1_ODE}
\dot{x}_1 & = -k_1 x_1 x_5 + k_3 x_3 x_4 \\
\dot{x}_2 & = k_1 x_1 x_5 - k_2 x_2 x_6 \\
\dot{x}_3 & = k_2 x_2 x_6 - k_3 x_3 x_4 \\
\dot{x}_4 & = k_1 x_1 x_5 - k_3 x_3 x_4 \\
\dot{x}_5 & = -k_1 x_1 x_5 + k_2 x_2 x_6 \\
\dot{x}_6 & = -k_2 x_2 x_6 + k_3 x_3 x_4\\
\end{split}
\end{align}
It can be checked from \eqref{eq:Ex1_ODE} that $\sum_{i=1}^6 x_i$ is a positive linear conserved quantity for the kinetic system \eqref{eq:Ex1_react}.
\subsection{Petri net representation of CRNs and persistence conditions}\label{subsec:PetriNet}
If we consider each molecule and reaction individually, CRNs can be described in the framework of discrete event systems \cite{pauleve2014dynamical}, and modeled e.g., by Petri nets \cite{angeli2007petri}. Moreover, certain properties of the corresponding Petri net have fundamental consequences on the continuous dynamics of the studied CRN. 

A \textit{Petri net} is a directed bipartite graph $G=(V,E)$, where $V=\{v_1,v_2,\dots,v_n \}$ is a set of vertices and $E=\{e_1,\dots,e_k\}$ is a set of directed edges, i.e. $e_i=(v_j,v_k)$, where $v_j, v_k \in V$. The set of vertices can be partitioned into two disjoint sets, the set of places denoted by $P=\{p_1,\dots,p_{n_p} \}$, and $T=\{t_1,\dots,t_{n_t} \}$ which is the set of transitions, where $P \cup T =V$, and $P \cap T = \emptyset$. Moreover, for any $e_i=(v_j, v_k)\in E$, either $v_j\in P$ and $v_k\in T$ or vice versa.

The state of a Petri net is given by the number of \textit{tokens} assigned to places. This can be characterized by a \textit{marking} $\mu: P \longrightarrow \mathbb{N}_0$. Obviously, a marking can be given as an integer vector of size $|P|$. The places from which edges point to a transition are called the \textit{input places of the transition}, while the places to which edges run from a transition are the \textit{output places of the transition}. The input and output places of a transition $t_i$ are denoted by $\text{In}(t_i)$ and $\text{Out}(t_i)$, respectively. Analogously, we can define the \textit{input} and \textit{output transitions} of a place $p_i$ denoted by $\text{In}(p_i)$ and $\text{Out}(p_i)$, respectively. Positive integers are assigned to each directed edge through the \textit{weighting} $W:(P \times T) \cup (T\times P) \longrightarrow \mathbb{N}$. A transition $t_j$  is \textit{enabled} if there are enough tokens in each of its input places: i.e., if $\forall$ $p_i\in\text{In}(t_j)$: $\mu(p_i) \ge W(p_i,t_j)$. During the \textit{firing} of an enabled transition $t_j$, $W(p_i,t_j)$ tokens are consumed from each $p_i\in\text{In}(t_j)$, and $W(t_j,p_k)$ tokens are added to each $p_k\in\text{Out}(t_j)$. 

The dynamical behaviour of a Petri net is characterized by the sequence of transitions from an initial marking $\mu_0$. Obviously, several transitions may be enabled at the same time which can fire in any order. Therefore, the execution (simulation) of Petri nets is generally nondeterministic. 

It can be seen from the above, that Petri nets can be assigned in a straightforward way to kinetic systems. In such a modeling framework, places correspond to species, and transitions represent reactions. The input and output places of a transition correspond to the species of the reactant and the product complexes, respectively. The weights of the input and output edges of a transition are the stoichiometric coefficients of the species of the reactant and product complexes, respectively.  For each place, the number of tokens show the actual number of molecules of the corresponding species. The Petri net representation of the CRN in Example 1 is shown in Fig. \ref{fig:Petri_triangle}, where species (places) and reactions (transitions) are denoted by circles and rectangles, respectively.
%
%

\begin{figure}[htbp]
	\centering
		\resizebox{5.5cm}{!}{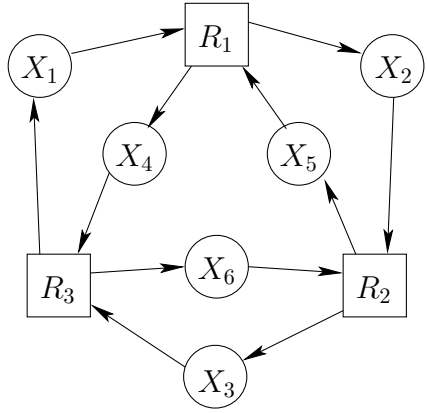}
		\caption{Petri net representation of the CRN described in Example 1. All edge weights are equal to 1.}\label{fig:Petri_triangle}
	
\end{figure}

A non-empty set of places $\sigma\subset P$ is called a \textit{siphon} if each input transition associated to $\sigma$ is also an output transition associated to $\sigma$. A siphon is \textit{minimal} if it does not contain (strictly) any other siphons. Naturally, the union of siphons is a siphon, too. With some abuse of notation, a set of species $\Sigma$ in a CRN will also be called a siphon if the places associated to the species of $\Sigma$ form a siphon in the Petri net of the reaction network.

We will use fundamental result from \cite{angeli2007petri} which can be re-written as follows.
\begin{thm}[Sufficient persistence conditions from \cite{angeli2007petri}]\label{thm:persistence}
The dynamics of a CRN of the form \eqref{eq:general_kinetic} is persistent if 
\begin{enumerate}
\item[(1)] There exists a positive linear conserved quantity $c^T x$ for the dynamics, where $c\in\mathbb{R}^n_{+}$.
\item[(2)] Each siphon of the CRN contains a subset of
species which define a positive linear conserved quantity for the dynamics.
\end{enumerate}
\end{thm}
The practical difficulty in applying Theorem 1 is that the number of siphons generally grows exponentially with the network size \cite{yamauchi1999time}, although there exist several computational approaches for the enumeration of all (minimal) siphons \cite{han2015calculation}.

Let us revisit Example 1 to illustrate the conditions of Theorem \ref{thm:persistence}. 
It is easy to see that condition (1) is fulfilled, since $I_0=\sum_{i=1}^6 x_i$ is a positive conserved quantity for the system as it was written in Subsection \ref{subsec:kinetic}. For condition (2), it can be checked from Fig. \ref{fig:Petri_triangle} that five minimal siphons exist in the Petri net, namely $\Sigma_1=\{X_1,X_2,X_3 \}$, $\Sigma_2=\{X_4,X_5,X_6 \}$, $\Sigma_3=\{X_1,X_4\}$, $\Sigma_4=\{X_2,X_5\}$, and $\Sigma_5=\{X_3,X_6\}$. Since $I_1=x_1+x_2+x_3$, $I_2=x_4+x_5+x_6$, $I_3=x_1+x_4$, $I_4=x_2+x_5$, and $I_5=x_3+x_6$ are also
positive linear first integrals containing the state variables of $\Sigma_1$,
$\Sigma_2$, $\Sigma_3$, $\Sigma_4$, and $\Sigma_5$, respectively, condition (2) is fulfilled, too. 

%
%
\section{The studied compartmental model class}\label{sec:CompModelClass}
In this paper, we consider a subclass of flow models equipped with a network structure. In such models, we have interconnected compartments and items (e.g., molecules, particles, or vehicles) moving between them. The compartments have finite capacities, i.e. we assume that there are well-defined upper limits for the number of items placed in the compartments at any time instant. The transition rate of items between two compartments depend on the number of particles in the source compartment and on the amount of available space in the target compartment. 

\subsection{Directed graph of the compartmental structure}
The structure of a compartmental model showing the possible directions of flows between the compartments can be described by a directed graph as follows.
\begin{definition}
The directed graph $D=(Q,A)$ called \textbf{compartmental graph}  describes the structure of the compartmental model, where the set $Q = \{q_1, \ldots , q_m \}$ of vertices correspond to the compartments, and the possible transitions are represented by directed edges of the set $A \subseteq Q \times Q$. The directed edge $a_{ij}:=(q_i,q_j) \in A$ represents the transition from the compartment $q_i$ into $q_j$. 
\end{definition}

Naturally, loop edges are not allowed in the compartmental graph, since the immediate transition from a compartment into itself does not induce any change. Furthermore, multiple identically directed edges are also not allowed between two compartments. If there exists a directed edge $(q_i,q_j)$ in the compartmental graph, then $q_i$ is called the \textit{donor} of compartment $q_j$, and $q_j$ is the \textit{recipient} of compartment $q_i$.

The directed graph of a strongly connected triangular compartmental model with $Q = \{q_1, q_2, q_3 \}$, and $A = \{(q_1, q_2), (q_2,q_3), (q_3,q_1) \} = \{a_{12}, a_{23}, a_{31} \}$ is shown in Fig. \ref{fig:compart_triangle}.
%
%

\begin{figure}[htbp]
	\centering
		\resizebox{4cm}{!}{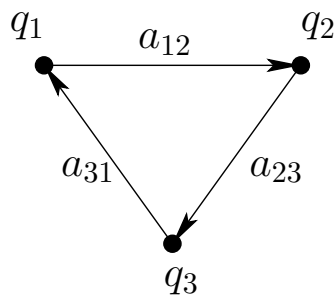}
		\caption{Directed graph of a simple triangular compartmental model}\label{fig:compart_triangle}
\end{figure}


\subsection{Kinetic representation of compartmental models}
We assign a CRN to a compartmental model $D=(Q,A)$ containing $m$ compartments as follows. The set of species is $\Sigma = \mathcal{N} \cup \mathcal{S}$, where  $\mathcal{N} = \{N_1, \dots N_m\}$ and $\mathcal{S}= \{S_1, \dots S_m \}$, where $N_i$ and $S_i$ represent the amount of particles and the available space in compartment $q_i$, respectively. To each transition (directed edge) $a_{ij}$ in $D$ we assign the following reaction
\begin{align}
N_i + S_j \autorightarrow{$\RR_{ij}$}{} N_j + S_i  \label{eq:trans_react}
\end{align}
Eq. \eqref{eq:trans_react} shows that during an elementary step of the transition from compartment $q_i$  to $q_j$, the amount of content (e.g., particles, material) in compartment $q_i$ is decreased by one unit, and the number of particles in compartment $q_j$ is increased by one unit. Parallelly, the amount of free space is increased in compartment $q_i$ and decreased in compartment $q_j$. It is also visible that a necessary condition for any transition is that there is at least one particle in compartment $q_i$ and at least one available space in $q_j$. The rate (velocity) of the transition is determined by the rate function $\RR_{ij}$. 
Let us denote the continuous amount (or concentration) of particles and free space in compartment $q_i$ by $n_i$ and $s_i$, respectively. Moreover, let $\mathcal{D}_i$ and $\mathcal{R}_i$ denote the index sets of the donor and recipient compartments of $q_i$, respectively. Then, using Eq. \eqref{eq:general_kinetic}, the dynamics of $n_i$ and $s_i$ can be written as
\begin{align}
\begin{split}\label{eq:comp_n_s}
\dot{n}_i & = \sum_{j\in\mathcal{D}_i} \RR_{ji} (n_j, s_i) - \sum_{j\in\mathcal{R}_i} \RR_{ji} (n_i, s_j) \\
\dot{s}_i & = - \sum_{j\in\mathcal{D}_i} \RR_{ji} (n_j, s_i) + \sum_{j\in\mathcal{R}_i} \RR_{ji} (n_i, s_j)
\end{split}
\end{align}
It is visible from Eq. \eqref{eq:comp_n_s} that $c_i:=n_i+s_i$ is constant for any compartment $q_i$, therefore, $c_i$ will be called the \textit{capacity} of $q_i$ for $i=1,\dots,m$.
From the CRN defined by the species and reactions in Eq. \eqref{eq:trans_react}, we can give the Petri net representation of a compartmental model as it is described in Subsection \ref{subsec:PetriNet}.

Consider the compartmental model shown in Fig. \eqref{fig:compart_triangle}. The associated CRN model is the following
\begin{align}
\begin{split}\label{eq:CRN_triangle}
N_1 + S_2 \autorightarrow{$\bar{\RR}_{12}$}{} N_2 + S_1 \\
N_2 + S_3 \autorightarrow{$\bar{\RR}_{23}$}{} N_3 + S_2 \\
N_3 + S_1 \autorightarrow{$\bar{\RR}_{31}$}{} N_1 + S_3 
\end{split}
\end{align}
It is easy to see that the CRN \eqref{eq:CRN_triangle} is identical to \eqref{eq:Ex1_react} in Example 1 with $N_i=X_i$ and $S_i=X_{i+3}$ for $i=1,2,3$, and $\bar{\RR}_{12} = \RR_1$, $\bar{\RR}_{23} = \RR_2$, $\bar{\RR}_{31} = \RR_3$. Therefore, the corresponding Petri net is the same as the one shown in Fig. \ref{fig:Petri_triangle}. Note that both the compartmental graph and the Petri net of the model are strongly connected. However, as it is visible from the disjoint complexes of the reactions listed in Eq. \eqref{eq:CRN_triangle}, the corresponding reaction graph is not weakly reversible and therefore not strongly connected.

It is important to remark that the model class introduced in this section includes as special cases certain finite volume discretizations of hyperbolic partial differential equations applied e.g., in (traffic) flow modeling \cite{kessels2019traffic,liptak2021traffic}, and popular kinetic models for the description of simultaneous mRNA translation and competition for ribosomes \cite{raveh2015ribosome,raveh2016model}.


%
\section{Persistence analysis}

\begin{prop}
If the compartmental graph $D$ of a model is strongly connected then the corresponding Petri net $\mathcal{P}(D)$ as a directed graph is also strongly connected. 
\end{prop}

A directed graph is strongly connected if for any two vertices $v$ and $w$ there is a directed path from $v$ to $w$. 
The idea of the proof is that if there are two paths $P_1=v_1 v_2 \ldots v_k$ and  $P_2=v_{k} v_{k+1} \ldots  v_n$, then by concatenating $P_2$ after $P_1$ we get a walk from $v_1$ to $v_n$. It will be a walk, and not necessarily a path, since there might be identical vertices in the two concatenated paths.
In this case by omitting the loop created between the first and last occurrences of a vertex we can get a shorter walk between the same endpoints. By a series of such steps the repeated occurences can be eliminated, and we get a path. Consequently, if there exists a walk from $v$ to $w$, then there exists a path from $v$ to $w$ as well, and for a graph to be strongly connected, it is enough to show the existence of walks instead of paths between any two vertices.   

\begin{proof} Recall from Section \ref{sec:CompModelClass} that the Petri net corresponding to the CRN representation of a compartmental model contains two types of vertices, $N_i$ and $S_j$ representing the number of molecules in compartment $q_i$ and the empty spaces in compartment $q_j$, respectively. The transition between the compartments $q_i$ and $q_j$ also corresponds to a vertex, but this is a different, reaction type vertex called $R_{ij}$. 

If in the compartmental graph $D$ there is a transition from compartment $q_i$ to $q_j$, then in the Petri net $\mathcal{P}(D)$ there exists a corresponding reaction vertex $R_{ij}$, which is incident to directed edges from  vertices $N_i$ and $S_j$, and to vertices $N_j$ and $S_i$. It can be seen in Figure \ref{PetriStructure} that a transition in the compartmental graph induces a path in the set $\mathcal{N}$ of the Petri net in the same direction, and in the set $\mathcal{S}$ in the opposite direction.

\begin{figure}[htbp]
	\centering
		\resizebox{10cm}{!}{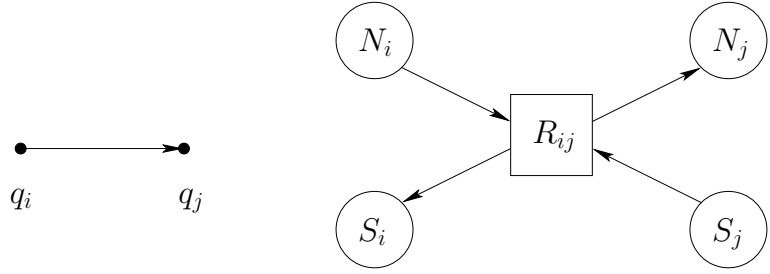}
		\caption{Representations of the same transition in the compartmental graph $D$ and in the Petri net $\mathcal{P}(D)$}
	\label{PetriStructure}
\end{figure}

We will examine the existence of paths between different types of vertices of the Petri net separately, i.e. there will be four cases:
\begin{enumerate}
\item Path from $N_i$ to $N_j$\\
Since the compartmental graph $D$ is strongly connected, it contains a directed path from $q_i$ to $q_j$:
\[q_i \rightarrow q_{k_1} \rightarrow q_{k_2} \rightarrow \ldots \rightarrow q_{k_l} \rightarrow q_j\] 

This implies a directed path in the Petri net $\mathcal{P}(D)$ from $N_i$ to $N_j$ with twice the length of the path in the compartmental graph:
\[N_i \rightarrow R_{i k_1} \rightarrow N_{k_1} \rightarrow R_{k_1 k_2} \rightarrow N_{k_2} \rightarrow\ldots \rightarrow N_{k_l} \rightarrow R_{k_l j} \rightarrow N_j\] 

\item Path from $S_i$ to $S_j$\\
The existence of such a path can be proven similarly as in the previous case. Because of the strong connectivity of the compartmental graph $D$ there exists a directed path from $q_j$ to $q_i$:
\[q_j \rightarrow q_{m_1} \rightarrow q_{m_2} \rightarrow \ldots \rightarrow q_{m_p} \rightarrow q_i\] 

This implies a directed walk in the Petri net $\mathcal{P}(D)$ from $S_i$ to $S_j$:
\[ S_j \leftarrow R_{j m_1} \leftarrow S_{m_1} \leftarrow R_{m_1 m_2}  \leftarrow S_{m_2}\leftarrow \ldots  \leftarrow S_{m_p} \leftarrow R_{m_p i} \leftarrow S_i\]

\item Path from $N_i$ to $S_j$\\
The strong connectivity of the compartmental graph $D$ implies that for every vertex there is at least one directed edge starting there, consequently there must be a vertex $q_k$ to which there is a transition from $q_i$. Through the reaction vertex in the Petri net there is a path $N_i \rightarrow R_{i k} \rightarrow S_i$. By concatenating this path with the existing path from $S_i$ to $S_j$ we get a walk from $N_i$ to $S_j$.

\item Path from $S_i$ to $N_j$\\
The strong connectivity of the compartmental graph $D$ also implies that for every vertex there is at least one directed edge going there, consequently there must be a vertex $q_l$ from which there is a transition to $q_i$. It follows that in the Petri net there is a path $S_i \rightarrow R_{l i} \rightarrow N_i$. By concatenating this path with the existing path from $N_i$ to $N_j$ we get a walk from $S_i$ to $N_j$.

\end{enumerate}
It can be seen that since there exist paths connecting the type $\mathcal{N}$ and type $\mathcal{S}$ vertices, then there are paths between the reaction type vertices as well.
\end{proof}

\begin{rem}
The implication in the other direction is not true. It is possible that the Petri net is strongly connected but the corresponding compartmental graph is not. Such an example can be seen in Figure \ref{StrongPetri}.
\end{rem}

\begin{figure}[htbp]
	\centering
		\resizebox{16cm}{!}{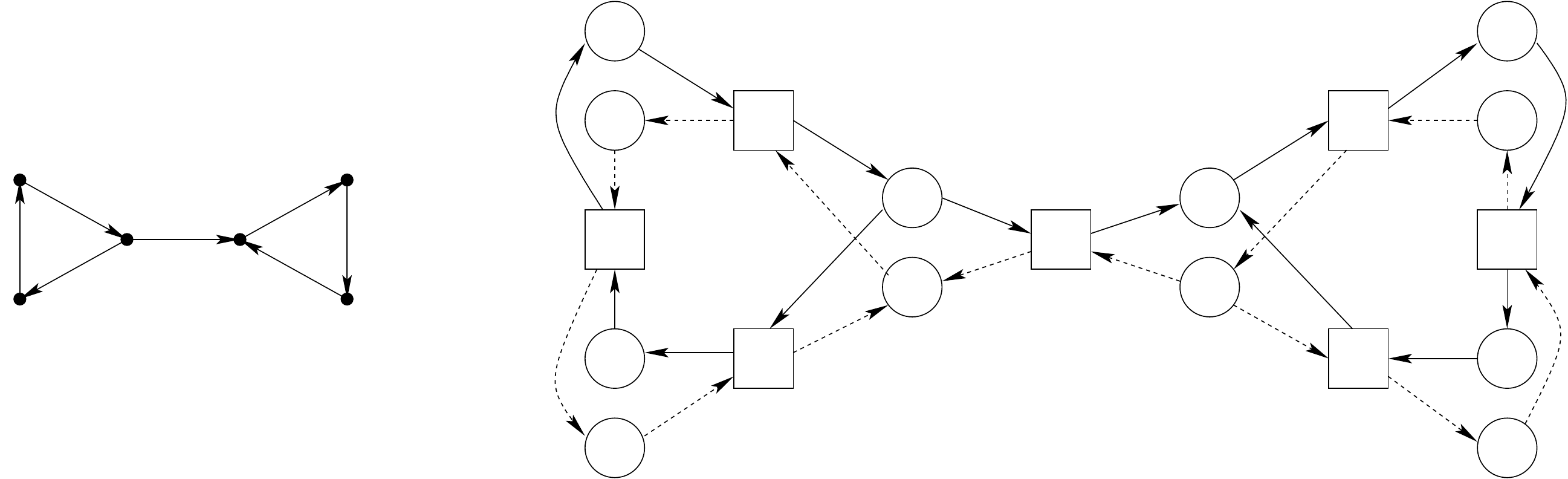}
		\caption{Strongly connected Petri net $\mathcal{P}(D)$ corresponding to a not strongly connected compartmental graph $D$}
	\label{StrongPetri}
\end{figure}

In the following part of this section we will examine the structure of siphons in the Petri net. According to the definitions it is a subset of  $\Sigma=\mathcal{N}\cup \mathcal{S}$ for which every input reaction is also an output reaction.

\begin{prop} \label{siphonN}
If the compartmental graph $D$  is  strongly connected and $\mathcal{N'} \subseteq \mathcal{N}$ is a siphon in the Petri net $\mathcal{P}(D)$, then $\mathcal{N'}= \mathcal{N}$ must hold.
\end{prop}

In other words, if a siphon of the Petri net contains vertices only from the set $\mathcal{N}$, then it contains all of them.

\begin{proof}
Let us assume by contradiction that $\mathcal{N}' $ is a siphon in the Petri net, which is a real subset of the vertex set $\mathcal{N}$. 

This set corresponds to the set $Q'=\{q_i ~|~ N_i \in \mathcal{N}'\}$, which is a real subset of the vertex set $Q$ of the compartmental graph $D$, i.e. $Q' \neq \emptyset$ and $Q \setminus Q' \neq \emptyset$ hold.
Since the compartmental graph $D$ is strongly connected, there must be vertices $q_i \in Q'$ and $q_j \in Q \setminus Q'$ so that there is a directed edge in $D$ from $q_j$ to $q_i$. 

This edge represents the transition corresponding to the reaction $R_{ji}$ in the Petri net, which connects the vertices $N_j, S_j, N_i$ and $S_i$. Since the vertex $N_i$ is in the siphon $\mathcal{N}'$, the reaction $R_{ji}$ is an input reaction of the siphon $\mathcal{N}'$, so by the definition of siphons, $R_{ji}$ must be an output reaction to $\mathcal{N}'$ as well. For this to hold $S_i$ or $N_j $ should be in the set $\mathcal{N}'$. However, $S_i \notin \mathcal{N}'$ since by its definition $\mathcal{N}'$ contains only type $\mathcal{N}$ vertices, and $N_j \notin \mathcal{N}'$ since $q_j \notin Q'$. The original assumption leads to contradiction, meaning that the siphon $\mathcal{N}'$ cannot be a real subset of $\mathcal{N}$. 

However, $\mathcal{N}$ itself is a siphon, since every reaction $R_{kl}$ is an input reaction to the vertex $N_l$ and an output reaction from the vertex $N_k$.

\end{proof}

A similar property is true for the species set $\mathcal{S}$, and the proof is based on the same idea.

\begin{prop} \label{siphonS}
If the compartmental graph $D$  is  strongly connected and $\mathcal{S'} \subseteq \mathcal{S}$ is a siphon in the Petri net $\mathcal{P}(D)$, then $\mathcal{S'}= \mathcal{S}$ must hold.
\end{prop}

If a siphon in the Petri net contains both types of vertices, a different type of structural property can be formulated.

\begin{prop} \label{siphon_mixed}
If the compartmental graph $D$  is  strongly connected and $\mathcal{T}\subseteq \mathcal{N} \cup \mathcal{S}$ is a siphon in the corresponding Petri net $\mathcal{P}(D)$ for which $\mathcal{T}\cap \mathcal{N} \neq \emptyset$ and $\mathcal{T}\cap \mathcal{S} \neq \emptyset$, then there is an index $i \in \{1,2,\ldots , m\}$ for which $N_i \in \mathcal{T}$ and $S_i \in \mathcal{T}$.
\end{prop}

In other words, if a siphon in the Petri net $\mathcal{P}(D)$ of a strongly connected compartmental graph $D$ contains type $\mathcal{N}$ and type $\mathcal{S}$ vertices as well, then there is a vertex $q_i$ in the compartmental graph $D$ for which the corresponding vertices $N_i$ and $S_i$ are both contained by the siphon.

\begin{proof}
Let us assume by contradiction that there is no such index. In this case the siphon $\mathcal{T}$ is of the form $\mathcal{N}' \cup \mathcal{S}'$, where the corresponding vertex subsets $Q_1=\{q_i ~|~ N_i \in \mathcal{N}'\}$ and $Q_2=\{q_j ~|~ S_j \in \mathcal{S}'\}$ of the compartmental graph $D$ are disjoint. Consequently, the sets $Q_1$, $Q_2$ and $Q \setminus (Q_1 \cup Q_2)$ form a partition of the vertices of $D$. (The set $Q \setminus (Q_1 \cup Q_2)$ might be empty, in this case we have a partition with two class instead of three.)

Since the compartmental graph $D$ is strongly connected, there must be an edge going into the set $Q_1$ from either or both of the other sets, i.e. there exist vertices $q_j \in Q_1$ and $q_i \in Q \setminus Q_1$ so that $q_i q_j$ is a directed edge in $D$. The corresponding reaction $R_{ij}$ in the Petri net $\mathcal{P}(D)$ is an input reaction to $N_j$ in the siphon $\mathcal{T}$, therefore it must be an output reaction as well. For this to be fulfilled, $S_j$ or $N_i$ should be in the siphon $\mathcal{T}$. By the assumption $\mathcal{T}$ cannot contain $S_j$ since $N_j \in \mathcal{T}$, and $\mathcal{T}$ cannot contain $N_i$ since the corresponding vertex $q_i$ is not in the set $Q_1$. This is a contradiction, consequently there must be an index $i$ for which $N_i \in \mathcal{T}$ and $S_i \in \mathcal{T}$ hold.
\end{proof}

\begin{cor}\label{cor:siphon_first_int} A siphon in the Petri net of a strongly connected compartmental graph either contains the $N_i$ and $S_i$ vertices corresponding to the same compartment $q_i$, or it contains all the vertices of the same type $\mathcal{N}$ or $\mathcal{S}$.
\end{cor}

\begin{prop} \label{siphon_i}
In the Petri net $\mathcal{P}(D')$ of any compartmental graph $D'$ the set $\{N_i, S_i\}$ is a siphon for all $i \in \{1,2,\ldots,m\}$.
\end{prop}

\begin{proof}
If $R$ is an input reaction to the set $\{N_i, S_i\}$, then it corresponds to a transition to or from the compartment $q_i$.

If $R$ is an input reaction to the vertex $N_i$ in the Petri net $\mathcal{P}(D')$, then it represents a transition $a_{ji}$ from some compartment $q_j$ to the compartment $q_i$ in the compartmental graph $D'$. In this case $R=R_{ji}$, and in the Petri net $\mathcal{P}(D')$ this reaction is also an output reaction from the vertex $S_i$.

If $R$ is an input reaction to the vertex $S_i$ in the Petri net $\mathcal{P}(D')$, then it represents a transition $a_{ik}$ from the compartment $q_i$ to some compartment $q_k$ in the compartmental graph $D'$. In this case $R=R_{ik}$, and in the Petri net $\mathcal{P}(D')$ this reaction is also an output reaction from the vertex $N_i$.


\end{proof}

\begin{figure}[htbp]
	\centering
		\resizebox{10cm}{!}{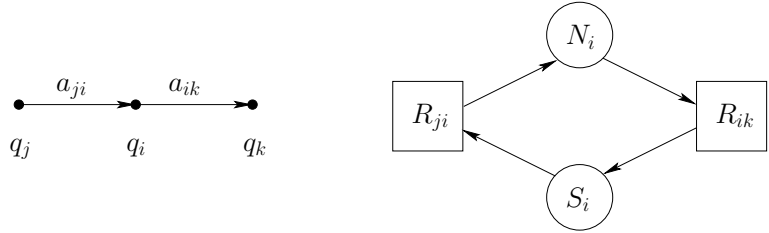}
		\caption{The vertex set $\{N_i, S_i\}$ is a siphon in the Petri net $\mathcal{P}(D')$.}
\end{figure}

\begin{cor} 
Among the siphons contained by the Petri net of a strongly connected compartmental graph the ones that are 
minimal with respect to containment are only the sets $\mathcal{N}$, $\mathcal{S}$ and the sets $\{N_i, S_i\}$ for all indices $i \in \{1,2, \ldots , m\}$.
\end{cor}

\begin{proof}
Let $\mathcal{T}$ refer to a siphon in the Petri net of a strongly connected compartmental graph.

If the siphon $\mathcal{T}$ contains just type $\mathcal{N}$ or just type $\mathcal{S}$ vertices, then by Propositions \ref{siphonN} and \ref{siphonS} it must contain all the vertices of that type. Consequently, the sets $\mathcal{N}$ and $\mathcal{S}$ are siphons that are minimal with respect to containment.

If the siphon $\mathcal{T}$ contains both type $\mathcal{N}$ and type $\mathcal{S}$ vertices, then by Proposition \ref{siphon_mixed} there must be an index $i \in \{1,2,\ldots , m\}$ for which $\{N_i,S_i\} \subseteq \mathcal{T}$ holds. In Proposition \ref{siphon_i} it was proven that the set $\{N_i,S_i\}$ for all indices is a siphon in the Petri net of a strongly connected compartmental graph. Consequently, if $\mathcal{T}$  is not equal to $\{N_i,S_i\}$ for some index $i$, then it is not minimal with respect to containment. 

A vertex $N_i$ in itself cannot be a siphon in the Petri net of a strongly connected compartmental graph, it follows from Proposition \ref{siphonN}. Similarly, $S_i$ in itself cannot be a siphon there, based on Proposition \ref{siphonS}. Consequently, for every index $i \in \{1,2,\ldots , m\}$ the set $\{N_i,S_i\}$ is a minimal siphon with respect to containment.

\end{proof}

\begin{cor}\label{cor:persistence}
	The dynamics given in Eq. \eqref{eq:comp_n_s} is persistent if the corresponding compartmental graph $D=(Q,A)$ is strongly connected. 
\end{cor}

\begin{proof}
	The statement is an immediate consequence of Theorem \ref{thm:persistence} and Corollary \ref{cor:siphon_first_int}.
\end{proof}

%
\section{Stability results}
Consider a system given by \eqref{eq:comp_n_s} with strongly connected compartmental structure. Using the fact that for each $q_i$ compartment $n_i=c_i-s_i$ we can rewrite the system in the reduced state-space as
\begin{equation}\label{eq:comp}
	\dot n_i=\sum_{j\in \mathcal{D}_i}\mathcal{K}_{ji}(n_j,c_i-n_i)-\sum_{j\in \mathcal{R}_i}\mathcal{K}_{ij}(n_i,c_j-n_j).
\end{equation}
The state-space of the above system is $C:=[0,c_1]\times[0,c_2]\times\dots\times[0,c_m]$ and let $\partial C$ denote the boundary of $C$; that is, $\partial C=C\backslash\mathrm{int}(C)$.

The Jacobian of \eqref{eq:comp} is given by
\begin{equation}
	\qty\big[J(n)]_{ik}=\begin{cases}
		-\sum_{j\in \mathcal{D}_i}\pdv{\mathcal{K}_{ji}(n_j,c_i-n_i)}{n_i}-\sum_{j\in \mathcal{R}_i}\pdv{\mathcal{K}_{ij}(n_i,c_j-n_j)}{n_i}\qquad&\text{if }i=k,\\
		\pdv{\mathcal{K}_{ki}(n_k,c_i-n_i)}{n_k}\qquad&\text{if }k\in \mathcal{D}_i\text{ and }k\not\in\mathcal{R}_i,\\
		\pdv{\mathcal{K}_{ik}(n_i,c_k-n_k)}{n_k}\qquad&\text{if }k\not\in \mathcal{D}_i\text{ and }k\in\mathcal{R}_i,\\
		\pdv{\mathcal{K}_{ki}(n_k,c_i-n_i)}{n_k}+\pdv{\mathcal{K}_{ik}(n_i,c_k-n_k)}{n_k}\qquad&\text{if }k\in \mathcal{D}_i\text{ and }k\in\mathcal{R}_i,\\
		0\qquad&\text{otherwise.}
	\end{cases}
\end{equation}
The (A2) property of the rate functions imply that each diagonal entry is nonpositive and each off-diagonal entry is nonnegative. Since the sum of each column is zero, we conclude that the system is compartmental in the sense of \cite{jacquez1993qualitative}. Systems satisfying the latter property are also called cooperative.

The following lemmata and proofs will adapt the ideas of \cite{margaliot2012stability} and \cite{raveh2015ribosome} for the studied more general system class. Moreover, we will also use the persistence result of Corollary \ref{cor:persistence}.

\begin{lemma}\label{lem:intC_remains}
	Consider a compartmental system of the form \eqref{eq:comp} with a strongly connected compartmental structure. Then, for any $n(0)\in\mathrm{int}(C)$ the solution satisfies $n(t)\in\mathrm{int}(C)$ for any $t\ge0$.
\end{lemma}

In other words, $\mathrm{int}(C)$ is an invariant set of such a system.

\begin{proof}
	To obtain a contradiction, suppose that there exists a (minimal) time $\tau>0$ such that $n(\tau)\in\partial C$. We need to consider the following two cases.
	\begin{enumerate}
		\item There exists an empty compartment. In this case, due to the strongly connected structure, there must exist an empty compartment with at least one non-empty donor compartment as well. To see this, consider a directed path from any non-empty compartment to any empty compartment. Stepping backwards from the empty compartment along this path until we reach a non-empty compartment establishes our assertion.

			Let $i$ be an index such that $n_i(\tau)=0$ and $n_k(\tau)>0$ holds for some $k\in \mathcal{D}_i$. Then \eqref{eq:comp} takes the form
			\begin{equation}
				\dot n_i(\tau)=\sum_{j\in \mathcal{D}_i}\mathcal{K}_{ji}(n_j,c_i)\ge\mathcal{K}_{ki}(n_k,c_i)>0
			\end{equation}
			which means that $\dot n_i(t)>0$ on the interval $[\tau-\sigma,\tau]$ for some $\sigma>0$. This leads to a contradiction with $n_i(\tau)=0$, further implying that there are no empty compartments altogether.
		\item There exists a full compartment. In this case, by a similar argument, there must exist a full compartment with at least one non-full recipient compartment as well; that is, there exists an index $i$ such that $n_i(\tau)=c_i$ and $n_k(\tau)<c_k$ holds for some $k\in \mathcal{R}_i$. Then \eqref{eq:comp} takes the form
			\begin{equation}
				\dot n_i(\tau)=-\sum_{j\in \mathcal{R}_i}\mathcal{K}_{ij}(c_i,c_j-n_j)\le-\mathcal{K}_{ik}(c_i,c_k-n_k)<0
			\end{equation}
			which means that $\dot n_i(t)<0$ on the interval $[\tau-\sigma,\tau]$ for some $\sigma>0$. This leads to a contradiction with $n_i(\tau)=c_i$, further implying that there are no full compartments altogether.
	\end{enumerate}
\end{proof}

Let $0^{(m)},c^{(m)}\in\mathbb{R}^m$ be defined by
\begin{equation*}
	0^{(m)}=\begin{bmatrix}0\\0\\\vdots\\0\end{bmatrix}\qquad c^{(m)}=\begin{bmatrix}c_1\\c_2\\\vdots\\c_m\end{bmatrix}.
\end{equation*}

\begin{lemma}\label{lem:goto_intC}
	Consider a compartmental system of the form \eqref{eq:comp} with a strongly connected compartmental structure. Then, for any $n(0)\in\partial C$, $n(0)\ne 0^{(m)}$, $n(0)\ne c^{(m)}$ the solution satisfies $n(\tau)\in\mathrm{int}(C)$ for some $\tau>0$.
\end{lemma}

\begin{proof}
	First we define the following boundary-repelling property.\newline

	(\textbf{BR}) For each $\delta>0$ and sufficiently small $\Delta>0$, there exists $K=K(\delta,\Delta)>0$ such that for each $t\ge0$
	\begin{enumerate}
		\item the conditions
			\begin{enumerate}
				\item $n_i(t)\le\Delta$,
				\item there exists $k\in \mathcal{D}_i$ such that $n_k(t)\ge\delta$
			\end{enumerate}
			imply $\dot n_i(t)\ge K$, and
		\item the conditions
			\begin{enumerate}
				\item $n_i(t)\ge c_i-\Delta$
				\item there exists $k\in \mathcal{R}_i$ such that $n_k(t)\le c_k-\delta$
			\end{enumerate}
			imply $\dot n_i(t)\le -K$.
	\end{enumerate}

	\eqref{eq:comp} satisfies the above property. To see this, consider any compartment $q_i$. Without the loss of generality we can assume that $\mathcal{D}_i$ contains at least one index, let this be $k$. In this case
	\begin{equation}
		\dot n_i(t)\ge\mathcal{K}_{ki}(\delta,c_i-\Delta)-\sum_{j\in \mathcal{R}_i}\mathcal{K}_{ij}(\Delta,c_j):=K_1.
	\end{equation}
	Similarly, we can assume that $\mathcal{R}_i$ contains at least one index, let this be $l$. In this case
	\begin{equation}
		\dot n_i(t)\le\sum_{j\in \mathcal{D}_i}\mathcal{K}_{ji}(c_j,\Delta)-\mathcal{K}_{il}(c_i-\Delta,c_l-\delta):=-K_2.
	\end{equation}
	The properties of the rate functions imply that for a sufficiently small $\Delta$ we have $K_1>0$ and $-K_2<0$, thus taking $K=\min\qty{K_1,K_2}$ concludes our assertion.\newline

	Next, we will show that for each compartment $n_i(\tau)>0$ holds for some $\tau>0$.

	Without the loss of generality we can assume that there exists an index $i$ such that $n_i(t)\ge\epsilon_0$ on the interval $[0,\tau]$ for some $\epsilon_0>0$ and $\tau>0$. Define $\tau_m=\frac{\tau}{m}$ and proceed by induction. For $k=1,2,\dots,m$ we will define an appropriate $\epsilon_k>0$ and show that the $k$th generation recipients of the compartment $q_i$ have particle concentration of at least $\epsilon_k$  on the interval $[k\tau_m,\tau]$.

	Pick any $j\in \mathcal{R}_i$ (first generation recipient) and sufficiently small $\Delta>0$, define $K=K(\epsilon_0,\Delta)$ and $\epsilon_1=\min\qty{\Delta,K\tau_m}$ and let $t_0\in[0,\tau_m]$ such that $n_j(t_0)\ge\epsilon_1$. Such a $t_0$ must exist, since assuming $n_j(t)<\epsilon_1\le\Delta$ for each $t\in[0,\tau_m]$ would imply via (\textbf{BR}) that $\dot n_j(t)\ge K$ for each $t\in[0,\tau_m]$. This further implies that $n_j(\tau_m)\ge n_j(0)+K\tau_m\ge\epsilon_1$. This leads to a contradiction with $n_j(\tau_m)<\epsilon_1$.

	Our next claim is that $n_j(t)\ge\epsilon_1$ for each $t\in[t_0,\tau]$ and in particular $[\tau_m,\tau]$. Conversely, suppose that there exists some $t_1\in(t_0,\tau]$ such that $\xi:=n_j(t_1)<\epsilon_1$ and define $\sigma=\min\qty{t\in(t_0,\tau):n_j(t)\le\xi}$. Since $n_j(\sigma)\le\xi<\epsilon_1\le\Delta$, (\textbf{BR}) shows that $\dot n_j(\sigma)\ge K$; that is, $\dot n_j(t)>0$ on the interval $[\sigma-\nu,\sigma]$ for some $\nu>0$. But this would imply that $n_j(\sigma-\nu)<n_j(\sigma)$, contradicting the minimality of $\sigma$.

	Define $K=K(\epsilon_1,\Delta)$ and $\epsilon_2=\min\qty{\Delta,K\tau_m}$ and repeat the above steps for the set $\mathcal{R}_j$ for $j\in\mathcal{R}_i$ (second generation recipients). In subsequent induction steps define $K=K(\epsilon_k,\Delta)$ and $\epsilon_{k+1}=\min\qty{\Delta,K\tau_m}$ and repeat the above for the $k$th generation recipients of the compartment $q_i$. Since the compartments are strongly connected after at most $m$ induction steps we conclude that $n_i(\tau)>0$ for each $i=1,2,\dots,m$.

	To show that $n_i(\tau)<c_i$ holds as well, consider the complementary system obtained by rewriting \eqref{eq:comp_n_s} using $s_i=c_i-n_i$ as
	\begin{equation}\label{eq:comp_comp}
		\dot s_i=-\sum_{j\in \mathcal{D}_i}\mathcal{K}_{ji}(c_j-s_j,s_i)+\sum_{j\in \mathcal{R}_i}\mathcal{K}_{ij}(c_i-s_i,s_j).
	\end{equation}
	Repeating the above steps for \eqref{eq:comp_comp} shows that $s_i(\tau)>0$, further implying that $n_i(\tau)<c_i$; that is, indeed $n(\tau)\in\mathrm{int}(C)$.
\end{proof}

\begin{rem}\label{rem:instant}
	The proof also shows that for each $\tau>0$ there exists $\epsilon(\tau)>0$ with $\epsilon(\tau)\rightarrow0$ as $\tau\rightarrow0$, such that $n(\tau)\in[\epsilon,c_1-\epsilon]\times[\epsilon,c_2-\epsilon]\times\dots\times[\epsilon,c_m-\epsilon]$; that is, even if the initial value is on $\partial C$ the orbit enters $\mathrm{int}(C)$ after an arbitrarily short time.
\end{rem}


\begin{rem}
	A similar argument shows that $\partial C$ only contains the two trivial equilibria corresponding to an empty and a full network.

	To see this, let us first assume that $n^*$ is an equilibrium and for a compartment $q_i$ we have $n_i^*=0$. Then, by \eqref{eq:comp}
	\begin{equation*}
		\dot n_i^*=\sum_{j\in \mathcal{D}_i}\mathcal{K}_{ji}(n_j^*,c_i)=0
	\end{equation*}
	which is only possible if $n_j^*=0$ for each $j\in \mathcal{D}_i$. Induction shows that $n^*=0^{(m)}$.

	Next, let us assume that for a compartment $q_i$ we have $n_i^*=c_i$. Then, by \eqref{eq:comp}
	\begin{equation*}
		\dot n_i^*=-\sum_{j\in \mathcal{R}_i}\mathcal{K}_{ij}(c_i,c_j-n_j^*)=0
	\end{equation*}
	which is only possible if $n_j^*=c_j$ for each $j\in \mathcal{R}_i$. Induction shows that $n^*=c^{(m)}$.
\end{rem}

For a given initial condition $a\in C$, let $\varrho(t,a)$ denote the solution at time $t$ with $\varrho(0,a)=a$; that is $\varrho(t,a)=n(t)$ with $n(0)=a$. Since the total number of particles is conserved, the function $I:\mathbb{R}^m\mapsto\mathbb{R}$ defined by $I(y)=\sum_{i=1}^my_i$ is a first integral. For $s\in\qty\Big[0,I(c^{(m)})]$ let $L_s\subset C$ be the level set of $I$; that is,
\begin{equation}\label{eq:CompClass}
	L_s=\qty\big{a\in C:I(a)=s}.
\end{equation}
Using the terminology of CRN theory \cite{Feinberg2019}, the level sets defined in Eq. \eqref{eq:CompClass} are also called \textit{stoichiometric compatibility classes}.

\begin{prop}\label{thm:asymp}
	Consider a compartmental system of the form \eqref{eq:comp} with a strongly connected compartmental structure. Then, for any $s\in\qty\Big[0,I(c^{(m)})]$ the set $L_s$ contains a unique steady state $e_s$ satisfying $\lim_{t\rightarrow\infty}\varrho(t,a)=e_s$ for any $a\in L_s$.
\end{prop}
\begin{proof}
	Since $L_0=\qty{0^{(m)}}$ and $\varrho(t,0^{(m)})=0^{(m)}$, the statement holds for an empty network with $e_0=0^{(m)}$. Similarly, since $L_{I(c^{(m)})}=\qty{c^{(m)}}$ and $\varrho(t,c^{(m)})=c^{(m)}$, the statement holds for a full network with $e_{I(c^{(m)})}=c^{(m)}$.

	Choose $s\in\qty\Big(0,I(c^{(m)}))$ and $a\in L_s$. By the strongly connected compartmental structure the Jacobian $J(n)$ is irreducible on $\mathrm{int}(C)$ but may become reducible on $\partial C$. However, Lemmata \ref{lem:intC_remains} and \ref{lem:goto_intC} along with Remak \ref{rem:instant} show that \eqref{eq:comp} has repelling boundary; that is, $\varrho(t,a)\in\mathrm{int}(C)$ after an arbitrarily short time even if $a\in L_s\cap\partial C$. As a consequence, \eqref{eq:comp} is a cooperative irreducible system evolving in $\mathrm{int}(C)$ admitting a first integral with positive gradient. The result \cite[Theorem 10.]{Mierczynski1995} shows that $L_s$ either has precisely one equilibrium that attracts the whole level set or has zero equilibria and each $\omega$-limit set of the level set is empty. However, by the boundedness of the sequence $\qty{\varrho(k,a):k=1,2,\dots}\subset\mathrm{int}(C)$ the Bolzano-Weierstrass theorem implies that there is a convergent subsequence; that is, the $\omega$-limit set of $a$ cannot be empty. Furthermore, Corollary \ref{cor:persistence} implies that $\omega(a)\cap\partial C=\emptyset$ and the proof is complete.
\end{proof}

In the proofs above we used the notion of cooperative systems directly, however, the underlying theory involves so-called (strongly) monotone systems, which in our case, is a direct consequence of cooperativity, as shown by our next result.

For two points $x,y\in\mathbb{R}^m$, let
\begin{align}
	x\le y\qquad&\text{if }x_i\le y_i\text{ for }i=1,2,\dots,m,\\
	x<y\qquad&\text{if }x\le y\text{ and }x\neq y,\\
	x\ll y\qquad&\text{if }x_i<y_i\text{ for }i=1,2,\dots,m.
\end{align}
\begin{prop}\label{prop:monotone}
	Consider a compartmental system of the form \eqref{eq:comp} with a strongly connected compartmental structure. Then, for any $s\in\qty\Big[0,I(c^{(m)})]$ and $a,b\in L_s$, the relation $a\le b$ implies $\varrho(t,a)\le\varrho(t,b)$ and $a<b$ implies $\varrho(t,a)\ll\varrho(t,b)$ for any $t>0$.
\end{prop}
\begin{proof}
	If $x$ or $y$ is equal to $0^{(m)}$ or $c^{(m)}$, then the statement trivially holds. In any other case, use the proof of Proposition \ref{thm:asymp} to conclude that \eqref{eq:comp} is a cooperative irreducible system evolving in a convex and open set, namely, $\mathrm{int}(C)$. The statement is a direct consequence of \cite[Theorem 1., Theorem 3.]{Mierczynski1995}.
\end{proof}

Our final result in this topic gives further insight into the qualitative behaviour of \eqref{eq:comp}.

\begin{prop}\label{prop:contractive}
	Consider a compartmental system of the form \eqref{eq:comp} with a strongly connected compartmental structure. Then, for any $a,b\in C$ initial values and $t\ge0$	
	\begin{equation*}
		\norm{\varrho(t,a)-\varrho(t,b)}_{\ell^1(\mathbb{R}^m)}\le\norm{a-b}_{\ell^1(\mathbb{R}^m)}.
	\end{equation*}
\end{prop}

In other words, using the usual $\ell^1(\mathbb{R}^m)$ norm, the distance of two trajectories at any given time cannot be larger than the distance of the initial values. In particular, if $b=e_{I(a)}$, then we find that the convergence to $e_{I(a)}$ is monotone. 

\begin{proof}
	By \cite[Chapter 2.2]{Vidyasagar1978} the induced matrix measure by the $\ell^1$ vector norm is
	\begin{equation}
		\mu(A)=\max_{i}\qty\Bigg{[A]_{ii}+\sum_{j\neq i}\qty\big|[A]_{ji}|}.
	\end{equation}
	Since $\mu\qty\big(J(n))=0$, the result \cite[Theorem 1.]{Russo2010} implies the assertion of the proposition.
\end{proof}

\begin{rem}
	It is straightforward to extend our persistence and stability results to systems with a weakly reversible compartmental graph, when the dynamics unfold into isolated subsystems having strongly connected compartmental graphs. Furthermore, some of the above results on the qualitative behaviour, for example the monotonicity in Proposition \ref{prop:monotone} and Proposition \ref{prop:contractive} can be extended to systems with arbitrary compartmental topology.
\end{rem}
%

%
\section{Conclusions}
The fundamental dynamical properties of a class of kinetic compartmental systems having finite capacities were studied in this paper. The transition of particles between compartments was described by a kinetic model. For persistence analysis, the Petri net representation of the CRN model was used. We showed that the Petri net of models having a strongly connected compartmental graph is also strongly connected. It was also shown that the dynamics of models with a strongly connected compartmental graph is persistent for a wide class of rate functions. This result is based on the characterization of siphons in the Petri net of the system and the corresponding conserved quantities. The persistence property was then used in the stability analysis of compartmental models. It was shown that for strongly connected compartmental models, a unique equilibrium point exists within each stoichiometric compatibility class, and this equilibrium is asymptotically stable within each compatibility class even if the initial conditions are on the boundary of the nonnegative orthant (except for the two trivial boundary equlibria). Further work will be focused on the control related application of our results.
\section*{Acknowledgements}
The authors acknowledge the support of the National Research, Development and Innovation Office (NKFIH) through grants no. 131545 and the Thematic Excellence Programme (TKP2020-NKA-11). The work of M. Vághy has been supported by the project ÚNKP-21-3-I-PPKE-60 of NKFIH.

\bibliographystyle{plain}
\bibliography{Compartmental_refs}

\end{document}